\documentclass{amsart}
\usepackage{geometry}
\usepackage[ansinew]{inputenc}
\usepackage{comment}
\usepackage{amssymb}
\usepackage{latexsym}
\usepackage{graphics}
\usepackage{epstopdf}
\usepackage{color}
\usepackage{fancybox}
\usepackage{fancyhdr}
\usepackage{amsmath,amsfonts}
 \usepackage[pagewise]{lineno}
\begin{document}

\setlength{\parskip}{0.3\baselineskip}

\newtheorem{theorem}{Theorem}
\newtheorem{corollary}[theorem]{Corollary}
\newtheorem{lemma}[theorem]{Lemma}
\newtheorem{proposition}[theorem]{Proposition}
\newtheorem{definition}[theorem]{Definition}
\newtheorem{remark}[theorem]{Remark}
\renewcommand{\thefootnote}{\alph{footnote}}


\newcommand{\dual}[2]{\langle{#1},{#2}\rangle}
\newcommand{\intO}[1]{\int_{\Omega}{#1}\, dx}

\newcommand{\R}{{\mathbb R}}
\newcommand{\N}{{\mathbb N}}

\title{\bf Gevrey Class: Plate Euler-Bernoulli and Membrane-like electric network with indirect fractional damping}
\author{Fredy Maglorio Sobrado  Su\'arez\\
 and \\
Filomena Barbosa Rodrigues Mendes
}

\date{}
\maketitle

\let\thefootnote\relax\footnote{{\it Email address:} {\rm fredy@utfpr.edu.br} (Fredy Maglorio Sobrado Suárez (Department of Mathematics, Federal University of Technology  of Paran\'a, Brazil)),  {\rm filomena@utfpr.edu.br} (Filomena 
	Barbosa Rodrigues Mendes (Department of Electrical Engineering,  The Federal University of Technology    of Paran\'a, Brazil)).}

\begin{center} {\bf Dedication:} This work is dedicated to Susana Suárez de Sobrado. \end{center}
\begin{abstract}
The emphasis in this paper is on  the Coupled System of a Kirchhoff-Love Plate Equation with the
Equation of a Membrane-like Electrical Network, where the coupling is of higher order given
by the Laplacian of the displacement velocity  $\gamma\Delta u_t$ and the Laplacian of the potential electric field $\gamma\Delta v_t $,  here only one of the equations is conservative, and
the other has dissipative properties. The mechanism was dissipative is given by an intermediate damping $(-\Delta)^\theta v_t$    between the potential electric  $\theta=0$  (frictional damping) and
the Laplacian of the electric potential for $\theta=1$ (damping Kelvin Voigt).  We show that $S(t)=e^{\mathbb{B}t}$ is
not analytic for $\theta\in [0, 1[$ and analytic for $\theta=1$, however $S(t)=e^{\mathbb{B}t}$ decays exponentially for 
$0\leq \theta\leq 1$ and $S(t)$ is of Gevrey  sharp class $s>\frac{1}{\theta}$   when the parameter $\theta$ lies in the interval $]0,1[$.
		 
\end{abstract}


\setcounter{equation}{0}

\section{Introduction}

In the literature, several mathematical models describe a single electrical network connecting piezoelectric actuators and/or transducers, see for example, \cite{ASIsolaFMP}, \cite{CMIsolaDDVescovo} or \cite{VidoliIsola}. In particular, in \cite{CMIsolaDDVescovo}, equations  (2b) and  (2c),   we have, for example, the equations of a second-order electric transmission line with zero order or second-order dissipation:\\
(S,Z) and (S,S)-network: second-order network with zeroth-order dissipation and second-order dissipation
\begin{equation}\label{(2b)}
v_{tt}-\beta_2\Delta v+\delta_0v_t=0\qquad\text{and}\qquad v_{tt}-\beta_2\Delta v-\delta_2\Delta v_t=0.
\end{equation}
Where $v(x, t)$ denotes the time-integral of the electric potential difference between the nodes and the ground.  Note that in the first equation of \eqref{(2b)} we have the frictional damping and in the second we have the viscous damping or Kelvin Voigt.

The motivation for this research was born from the  coupled system of the Euler-Bernoulli  Plates and Membrane-Like Electric Network deduced in \cite{VidoliIsola} as follows:
\begin{eqnarray}
\label{ISplacas0}
    u_{tt}+\alpha\Delta^2 u-\gamma\Delta v_t=0,\quad
    &x\in\Omega,&
t>0, \\
\label{ISplacas1}
    v_{tt}
    -\beta\Delta v+\gamma\Delta u_t+\delta v_t+\delta\gamma\Delta u =0,\quad
    &x\in\Omega,&
     t>0,
\end{eqnarray}

satisfying the  boundary conditions
\begin{equation}\label{ISplacas4}
    u=\Delta u=0,\quad v=0,\quad x\in\partial\Omega,\ t>0,
\end{equation}
and prescribed initial data
\begin{eqnarray}
\label{ISplacas2} u(x,0)=u_0(x),\ u_t(x,0)=u_1(x),\
v(x,0)=v_0(x),\ v_t(x,0)=v_1(x), \quad x\in\Omega.
\end{eqnarray}
Here,  $u(x,t)$ denotes the transversal displacements of the plates and $v(x,t)$ is time- integral of the electric potential difference between the nodes and the ground, and $\Omega\subset \R^n$  the domain with smooth boundary $\partial\Omega$. The coefficients $\alpha,\beta, \delta$ are positive and $\gamma$ is non-zero, for details of the physical meaning and as determined each of the coefficients consult the deduction of the Physical-Mathematical model on pages $441$ and $442$ of reference \cite{VidoliIsola}. For more details on modeling, the reference \cite{ASIsolaFMP} can also be consulted.

Our purpose in this work is to study a more general system, to this end, we will consider in the equation of the electrical network the fractional dissipation $(-\Delta)^\theta v_t $ for  $ 0\leq\theta\leq 1 $, keep in mind that for the particular cases $\theta= 0$ and $\theta=1$ the mathematical models are given by equations in  \eqref{(2b)}   of  \cite{VidoliIsola}  respectively.

 We will write the system under study in its abstract form. For this purpose, we introduce some helpful notations beforehand.  Let $\Omega$ a bounded set in $\R^n$ with smooth boundary and given the operator:
$A:D(A)\subset L^2(\Omega)\to L^2(\Omega)$, where
\begin{eqnarray}\label{Omenoslaplaciano}
A=-\Delta,\quad D(A)=H^2(\Omega)\cap H^1_0(\Omega).
\end{eqnarray}
It is known that this operator given in \eqref{Omenoslaplaciano} is selfadjoint, positive, compact inverse, and compact resolvent. Using this $A$ operator, our proposed system, written abstractly,  is as follows:
\begin{eqnarray}
\label{ISplacas-20}
    u_{tt}+ \alpha A^2 u+\gamma Av_t=0, \quad x\in\Omega,\quad t>0,\\
\label{ISplacas-25}
    v_{tt}+ \beta A v-\gamma Au_t+\delta A^{\theta} v_t =0,\quad x\in\Omega,\quad t>0,
\end{eqnarray}
and contemplates the boundary conditions \eqref{ISplacas4} and initial data \eqref{ISplacas2}.

In the last decades, many researchers have focused on studying the asymptotic stability of several coupled systems with indirect damping (Terminology initially used by Russell in his work \cite{Russell-1993}). Systems of two coupled equations as wave-wave, plate-plate, or plate-wave equations with indirect damping inside of their domains or on their boundaries, were studied by several authors. We are going  briefly mention some of these works:

 Alabau et al.  in \cite{Alabau-2002}. They considered abstract evolution equations given by:
\begin{eqnarray*}
	u_{tt}+ A_2 u+\alpha v=0,  \quad x\in\Omega,\quad t>0,\\
	v_{tt}+ A_1 v+\beta Bv_t+\alpha u=0, \quad x\in\Omega,\quad t>0,
\end{eqnarray*}
in which  $\Omega$ be a bounded open set of $\R^n$ with smooth boundary $\partial\Omega$ and $A_1$, $A_2$ are self-adjoint positive linear operators in Hilbert space and $B$ is a bounded operator. When $A_1=-\Delta=A, A_2=\Delta^2$ and $B$ is the identity operator, we have a wave-Petrowsky system,
where $\beta > 0$, with partial frictional damping $\beta u_t$. For this case, they showed that, if $0<|\alpha|<C_{\Omega}^{3/2}$ and 
$$v_0\in H^3(\Omega)\cap H_0^2(\Omega), \quad u_0\in H^6(\Omega)\cap H_0^3(\Omega),$$
$$v_1\in H^2(\Omega)\cap H_0^1(\Omega),\quad u_1\in H^4(\Omega)\cap H_0^2(\Omega).$$
Then the energy of the solution satisfies, for every $t>0$, the estimate
\begin{gather*}
\hspace*{-7cm}	\int_\Omega(|\partial_tv|^2+|\nabla v|^2+|\partial_t u|^2+|\Delta u|^2)dx \\
\leq 
 \dfrac{C}{t}(\|v_0\|^2_{3,\Omega}+\|u_0\|_{6,\Omega} +\|v_1\|^2_{2,\Omega}+\|u_1\|^2_{4,\Omega}).
\end{gather*}
 In this direction, other results can be found in \cite{Alabau-2011, EngelPruss-2006, Guglielmi2015,Hao-2015,Suarez}.

Alabau et al. \cite{Alabau-2011} (see also \cite{Alabau1999, Alabau-2002, Alabau2002-2}) considered an abstract system of two coupled evolution equations with applications to several hyperbolic systems satisfying hybrid boundary conditions. They have shown their solutions' polynomial decay using energy and multiplicative techniques. Tebou \cite{Tebou2012} considered a weakly coupled system of plate-wave equations with indirect frictional damping mechanisms. He showed this system is not exponentially stable when showed the damping acts either in the plate equation or in the wave equation, and a polynomial decay of the semigroup using a frequency domain approach combined with multiplier techniques, and a recent Borichev and Tomilov\cite{Borichev} result in the characterization of polynomial decay of bounded semigroups. Recently, Guglielmi \cite{Guglielmi2015} considered two classes of systems of weakly coupled hyperbolic equations wave-wave equation and a wave-Petrovsky system. When the wave equation is frictionally damped, he proved that this system is not exponentially stable, and a polynomial decay was obtained. Provided o result of the optimal decay rate was provided.    Many other papers were published in this direction; viewed in \cite{OquendoRaya2017, Renardy,Tebou2012, Tebou-2017}.

Now we will mention some concrete problems that motivated the work in  of this paper:

Han and Liu in \cite{HanLiu} have recently studied the regularity and asymptotic behavior of two-plate system solutions where only one of them is dissipative and indirect system dissipation occurs through the higher order coupling term $\gamma\Delta w_t$ and $-\gamma\Delta u_t$. The damping mechanism considered in this work was structural or Kelvin-Voigt damping. More precisely, the system studied in \cite{HanLiu} is:
\begin{eqnarray*}
	u_{tt}+\Delta^2 u+\gamma\Delta w_t=0,\quad
	&x\in\Omega,&
	t>0, \\
	w_{tt}+\Delta^2 w-\gamma \Delta u_t-d_{st}\Delta  w_t+d_{kv}\Delta^2w_t =0,\quad
	&x\in\Omega,&
	t>0,
\end{eqnarray*}
satisfying the  boundary conditions
\begin{equation*}
u=\frac{\partial u}{\partial\nu}=0,\quad w=\frac{\partial w}{\partial\nu}=0,\quad \ t>0,\quad x\in \partial\Omega,
\end{equation*}
where $u(x,t)$, $w(x,t)$ denote the transversal displacements of the plates at time $t$ in the domain $\Omega\subset \R^n$ with smooth boundary $\partial\Omega$, $\gamma\not=0$ is the coupling coefficient. 
They showed that if $d_{st}>0 $ and $d_{kv}=0$, the semigroup associated with the system is analytic and for $d_{st}=0$ and $d_{kv}>0$, they showed that $ S(t)$ is exponential but not analytic.

In 2013, Dell'Oro et al. in \cite{Dell'Oro}. They considered the abstract system  with fractional partial  damping:
\begin{eqnarray*}
	u_{tt}+\gamma A u_{tt}+ A^2 u-A^{\sigma}\phi=0,  \quad x\in\Omega,\quad t>0,\\
	\phi_{t}+ A \phi+A^{\sigma} u_t =0, \quad x\in\Omega,\quad t>0,
\end{eqnarray*}
where $\Omega$ be a bounded open set of $\R^n$ with smooth boundary $\partial\Omega$ and when $A=-\Delta$ as in \eqref{Omenoslaplaciano} this system models a  thermoelastic plate, where the parameter $\gamma\geq 0$ is responsible for the rotational inertia,
which is proportional to the plate thickness, $\gamma=0$, corresponding to the case of a thin plate.  They showed that the semigroup of this system is exponentially stable if and only if $\sigma\geq 1$. Moreover, when $1/2\leq \sigma<1$, they proved that the semigroup decays polynomially to zero as $t^{-1/(4-4\sigma)}$ for initial data in the domain of the semigroup generator, and such a decay rate is optimal. In this same work, they also showed that for the case $\gamma=0$ and $0\leq \sigma<1/2$, the semigroup decays polynomially with the optimal rate $t^{-1/(1-2\sigma)}$. Other results in this direction can be found in \cite{EngelPruss-2006,Shibata,Tebou-2010,Tebou-2013}.

A more recent result involving fractional dissipation was published in 2019 by Oquendo-Suárez \cite{HPOquendo}, they studied the following abstract system:
\begin{eqnarray*}
\rho_1u_{tt}+\gamma_1 A u_{tt}+\beta_1 A^2 u+\alpha v=0,  \quad x\in\Omega,\quad t>0,\\
\rho_2v_{tt}+\gamma_2 A v_{tt}+\beta_2 A^2 v+ \alpha u+\kappa A^{\theta} v_t =0, \quad x\in\Omega,\quad t>0,
\end{eqnarray*}
 where $\Omega$ be a bounded open set of $\R^n$ with smooth boundary $\partial\Omega$ and one of these equations is conservative and the other has fractional dissipative properties given by $A^{\theta} v_t $, where $0\leq\theta\leq 1$ and $A=-\Delta$ as in \eqref{Omenoslaplaciano}   and where the coupling terms are $\alpha u $ and $\alpha v $.  They showed that the  semigroup decays polynomially with a rate that depends on $\theta$ and some relations between the structural coefficients of the system. Have also shown that the rates obtained are optimal using a spectral characterization theorem of semigroup polynomial stability due to Borichev and Tomilov \cite{Borichev}.
 
Concerning the regularity of the semigroup associated with plate models, we can cite the work of \cite{LMJAIME2012} of 2012; in that work, the authors study the differentiability and analyticity of the associated semigroup and also determine the optimum rate of decay and more recently published works explore the regularity of solutions using the Gevrey classes introduced in 1989 in the thesis of Taylor \cite{TaylorM}. Among these works, we can mention Hao-Liu-Yong \cite{Hao-2015} and, more recently, the paper of Keyantuo-Tebou-Warma \cite{Tebou-2020} to be published. In this last work, the authors studied the thermoelastic plate model with a fractional Laplacian between the Euler-Bernoulli and Kirchhoff model with two types of boundary conditions; in addition to studying the asymptotic and analytical behavior, the authors show that the underlying semigroups are of Gevrey class $\delta$ for every $\delta>\frac{2-\theta}{2-4\theta}$ for both the clamped and hinged boundary conditions when the parameter $\theta$ lies in the interval $]0,1/2[$.

This article was organized as follows: In section 2, we study the well-posedness of the system \eqref{ISplacas-20}-\eqref{ISplacas-25} through the semigroup theory. We left our main results for the last two sections. In Section 3, we prove the exponential decay of the semigroup $S(t)=e^{\mathbb{B}t}$, for  $0\leq\theta\leq 1$. Section 4 deals with the lack of analyticity of the semigroup $S(t)=e^{\mathbb{B}t}$ for $\theta\in [0,1[$  and analyticity de $S(t)$ for $\theta=0$; in particular, we address the case $0\leq \theta<1$ in subsection 4.1, while the case $\theta=1$ is discussed in subsection 4.2. Finally in section 5 we show that $S(t)=e^{\mathbb{B}t}$ is of Gevrey sharp class  $s>\frac{1}{\theta}$ when the parameter $\theta$ lies in the interval $]0,1[$.

\section{Well-Posedness of the System}
We will use a semigroup approach to show the existence uniqueness of strong solutions for the abstract system  \eqref{ISplacas-20}-\eqref{ISplacas-25}. It is important recalling that $A$  defined in \eqref{Omenoslaplaciano}  is a positive self-adjoint operator with compact inverse on a complex Hilbert space $D(A^0)=L^2(\Omega)$. Therefore, the operator $A^{\theta}$ is self-adjoint positive for  all $\theta\in\R$ and the embedding
\begin{eqnarray*}
D(A^{\theta_1})\hookrightarrow D(A^{\theta_2}),
\end{eqnarray*}
is continuous for $\theta_1>\theta_2$. Here, the norm in $D(A^{\theta})$ is given by $\|u\|_{D(A^{\theta})}:=\|A^{\theta}u\|$, $u\in D(A^{\theta})$, where $\|\cdot\|_\mathcal{H}$ denotes the norm in the Hilbert space $\mathcal{H}$. Some of these spaces are: $D(A^{1/2})=H_0^1(\Omega)$, $D(A^0)=L^2(\Omega)$ and $D(A^{-1/2})=H^{-1}(\Omega)$. 

Now, we will use a semigroups approach to study the well-posedness of the system \eqref{ISplacas-20}-\eqref{ISplacas-25}.  Taking $w=u_t$, $v_t=z$ and  considering $U=(u,v,w,z)$ and $U_0=(u_0,v_0,u_1,v_1)$, the system \eqref{ISplacas-20}--\eqref{ISplacas-25}, can be written in the following abstract framework
\begin{equation}\label{Fabstrata}
    \frac{d}{dt}U(t)=\mathbb{B} U(t),\quad    U(0)=U_0,
\end{equation}
 where the operator $\mathbb{B}$ is given by
  \begin{gather} \label{operadorAgamma}
  \mathbb{B}U:=\Big(w,\ z,\ - \alpha A^2 u-\gamma A z,- \beta A v+\gamma Aw-\delta A^{\theta} z \Big),
  \end{gather}

for $U=(u,v,w,z)$. This operator will be defined in a suitable subspace of the phase space
$$              \begin{array}{ll}
               \mathcal{H}:= D(A)\times D(A^\frac{1}{2})\times D(A^0)\times D(A^0).
              \end{array}
$$
It's  a Hilbert space with the inner product
$$\langle U_1,U_2\rangle  :=\alpha  \dual{A u_1}{A u_2}+\beta \dual{A^\frac{1}{2} v_1}{A^\frac{1}{2} v_2}+ \dual{ w_{1}} {w_{2}}+\dual{ z_{1}} {z_{2}},
$$
for $U_i=(u_i, v_i, w_i, z_i)\in \mathcal{H}$,  $i=1,2$., and we endow it with the norm given by
\begin{equation}\label{NORMA}
\|U\|^2_\mathcal{H}:=\alpha\|Au\|^2+\beta\|A^\frac{1}{2}v\|^2+\|w\|^2+\|z\|^2.
\end{equation}

 In these conditions, we define the domain of $\mathbb{B}$ as
\begin{align*}
    \mathcal{D}(\mathbb{B}):= 
    \Big\{ U\in \mathcal{H}\colon (w,z)\in  D(A)\times D(A^\frac {1}{2} ),&(-\alpha A u-\gamma z, -\beta v-\delta A^{\theta-1}z)\in [D(A)]^2\Big\}.\label{dominioA}
\end{align*}

To show that the operator $\mathbb{B}$ is the generator of a $C_0$- semigroup we invoke a result from Liu-Zheng' book.

\begin{theorem}[see Theorem 1.2.4 in \cite{LiuZ}] \label{TLiuZ}
Let $ \mathbb{B}$ be a linear operator with domain $\mathcal{D}(\mathbb{B})$ dense in a Hilbert space $\mathcal{H}$. If $ \mathbb{B}$ is dissipative and $0\in\rho( \mathbb{B})$, the resolvent set of $ \mathbb{B}$, then $ \mathbb{B}$ is the generator of a $C_0$- semigroup of contractions on $\mathcal{H}$.
\end{theorem}

Let us see that the operator $\mathbb{B}$ in \eqref{operadorAgamma}  satisfies the conditions of this theorem. Clearly, we see that $\mathcal{D}(\mathbb{B})$ is dense in $\mathcal{H}$. Effecting the internal product of $\mathbb{B}U$ with $U$, we have
\begin{equation}\label{eqdissipative}
\text{Re}\dual{\mathbb{B}U}{U}=  -\delta\|A^{\theta/2} z\|^2, \quad\forall\ U\in \mathcal{D}(\mathbb{B}),
\end{equation}
that is, the operator $\mathbb{B}$ is dissipative.

To complete the conditions of the above theorem, it remains to show that $0\in\rho(\mathbb{B})$. Let $F=(f_1,f_2,f_3,f_4)\in \mathcal{H}$, let us see that the stationary problem $ \mathbb{B}U=F$ has a solution $U=(u,v,w,z)$.  From the definition of the operator  $\mathbb{B}$ given in
\eqref{operadorAgamma}, this system
can be written as                     
\begin{align}
	w=f_1,\qquad& \quad\quad  \alpha A^2 u=-[\gamma Af_2+f_3], \label{exist-10}\\
	z=f_2,\qquad &  \quad\quad \beta Av =\gamma Af_1-\delta A^\theta f_2-f_4. \label{exist-20}
\end{align}
This problem can be placed in a variational formulation: to find $t=(u,v)$ such that
\begin{eqnarray}\label{var-10}
b(t,z)=h(z):=\dual{h}{z},\quad\forall\ z=(z_1,z_2)\in D(A)\times D(A^\frac{1}{2}),
\end{eqnarray}
where\\
 $h=(-[\gamma Af_2+f_3],\gamma Af_1-\delta A^\theta f_2-f_4) \in D(A^0)\times D(A^0)$ and
\begin{eqnarray*}
b(u,v; z_1,z_2):=\alpha\dual{Au}{Az_1}+\beta\dual{A^\frac{1}{2}v}{A^\frac{1}{2}z_2}.
\end{eqnarray*}
Consequently
\begin{equation}\label{coercivity}
b(t,t)=\alpha\|Au\|^2+\beta\|A^\frac{1}{2}v\|^2.
\end{equation}

Of \eqref{coercivity}   the proof of the coercivity of this sesquilinear form $b$ in Hilbert space $D(A)\times D(A^\frac{1}{2})$ is immediate, now, applying the Lax-Milgram Theorem and taking into account the first equations of \eqref{exist-10}-\eqref{exist-20} we have a unique solution $U\in \mathcal{H}$. As this solution satisfies the system (\ref{exist-10})-(\ref{exist-20}) in a weak sense, from these equations we can conclude that $U\in \mathcal{D}(\mathbb{B})$.

Again, from \eqref{coercivity} and the second equations of \eqref{exist-10}-\eqref{exist-20},    applying Cauchy-Schwarz and Young inequalities to the second member of this inequality, for $\varepsilon>0$  there exists $K_\varepsilon >0$, such that
\begin{eqnarray*}
\alpha\|A u\|^2+ \beta\|A^\frac{1}{2} v\|^2\leq C_{\varepsilon}\|F\|^2.
\end{eqnarray*}
This inequality and the first equations of (\ref{exist-10})-(\ref{exist-20}) imply that $\|U\|_\mathcal{H}\leq C\|F\|_\mathcal{H}$, then  $0$ belongs to the resolvent set $\rho(\mathbb{B})$. Consequently, from Theorem \ref{TLiuZ}  we have  $\mathbb{B}$ as the generator of a contractions semigroup.

As $\mathbb{B}$ is the generator of a $C_0$-semigroups the solution of the abstract system (\ref{Fabstrata}) is given by $U(t)=e^{t\mathbb{B}}U_0$, $t\geq 0$. Thus, we have shown the following well-posedness theorem:
\begin{theorem}[see \cite{Pazy}] Let us take initial data $U_0$ in $\mathcal{H}$ then there exists only one solution to the problem (\ref{Fabstrata}) satisfying
$$
U\in C([0,\infty[;\mathcal{H}).
$$
Moreover,  if $U_0\in D(\mathbb{B})$ then the solution satisfies
$$
U\in C([0,\infty[;\mathcal{D}(\mathbb{B}))\cap C^1([0,\infty[;\mathcal{H}).
$$
\end{theorem}


\section{Stability Results}
In this section, we will study the asymptotic behavior of the semigroup of the system \eqref{ISplacas-20}-\eqref{ISplacas-25}. First, we will use the following spectral characterization of exponential stability of semigroups due to Gearhart\cite{Gearhart}(Theorem 1.3.2  book of Liu-Zheng ), and to study analyticity we will use a characterization of the book of Liu-Zheng (Theorem 1.3.3).
\begin{theorem}[see \cite{LiuZ}]\label{LiuZExponential}
Let $S(t)=e^{\mathbb{B}t}$ be  a  $C_0$-semigroup of contractions on  a Hilbert space $ \mathcal{H}$. Then $S(t)$ is exponentially stable if and only if  
	\begin{equation}\label{EImaginario}
\rho(\mathbb{B})\supseteq\{ i\lambda/ \lambda\in \R \} 	\equiv i\R
\end{equation}
and
\begin{equation}\label{Exponential}
 \limsup\limits_{|\lambda|\to
   \infty}   \|(i\lambda I-\mathbb{B})^{-1}\|_{\mathcal{L}( \mathcal{H})}<\infty
\end{equation}
holds.
\end{theorem}
\begin{theorem}[see \cite{LiuZ}]\label{LiuZAnaliticity}
	Let $S(t)=e^{\mathbb{B}t}$ be $C_0$-semigroups of contractions  on a Hilbert space $ \mathcal{H}$. Suppose that
	\begin{equation*}
	\rho(\mathbb{B})\supseteq\{ i\lambda/ \lambda\in \R \} 	\equiv i\R
	\end{equation*}
	 Then $S(t)$ is analytic if and only if
	\begin{equation}\label{Analiticity}
	 \limsup\limits_{|\lambda|\to
		\infty}
	\|\lambda(i\lambda I-\mathbb{B})^{-1}\|_{\mathcal{L}( \mathcal{H})}<\infty
	\end{equation}
	holds.
\end{theorem}


In what follows: $C$, $C_\delta$, $C_{\delta_1}$ and $K_\varepsilon$ will denote positive constants that assume different values in different places, and the coupling coefficient $\gamma$  will be assumed positive (the results remain valid when this coefficient $\gamma$ is negative).

 First, note that if $\lambda\in\R$ and $F=(f_1,f_2,f_3,f_4)\in \mathcal{H}$ then the solution $U=(u,v,w,z)\in\hbox{D}(\mathbb{B
 })$ of the stationary system $(i\lambda I- \mathbb{B
 })U=F$ can be written in the form
\begin{eqnarray}
i\lambda u-w &=& f_1,\label{esp-10}\\
i\lambda v-z &=& f_2,\label{esp-20}\\
i\lambda  w+ \alpha A^2 u+\gamma Az &=&f_3,\label{esp-30}\\
i\lambda  z+\beta A v-\gamma Aw+\delta A^{\theta} z&=& f_4.\label{esp-40}
\end{eqnarray}
We have 
\begin{eqnarray}\label{dis-10}
\delta\|A^{\frac{\theta}{2}}z\|^2=\text{Re}\dual{(i\lambda -\mathbb{B})U}{U}=\text{Re}\dual{F}{U}\leq \|F\|_\mathcal{H}\|U\|_\mathcal{H}.
\end{eqnarray}
From equations (\ref{esp-20}) and \eqref{dis-10}, we have
\begin{eqnarray}\label{dis-10A}
|\lambda|^2\|A^{\frac{\theta}{2}}v\|^2 &\leq& C\{\|F\|_\mathcal{H}\|U\|_\mathcal{H}+\|F\|^2_\mathcal{H}\}.
\end{eqnarray}
As $\frac{\theta-2}{2}\leq  0\leq\frac{\theta}{2}$, taking into account the continuous embedding $D(A^{\theta_2})\hookrightarrow D(A^{\theta_1})$, $\theta_2>\theta_1$ and \eqref{dis-10}, we obtain
\begin{eqnarray}
\label{dis-10B}
  \|A^\frac{\theta-2}{2}z\|^2 &\leq&C\{\|F\|_\mathcal{H}\|U\|_\mathcal{H}+\|F\|^2_\mathcal{H}\}. \\
\label{dis-10C}
 \|z\|^2 &\leq&C\{\|F\|_\mathcal{H}\|U\|_\mathcal{H}+\|F\|^2_\mathcal{H}\}.
\end{eqnarray}

\subsection{Exponential Decay of $S(t)$ for $0\leq\theta\leq 1$}
In this subsection, we show the exponential decay using Theorem \ref{LiuZExponential},  to demonstrate condition  \eqref{Exponential}.  Just demonstrate
\begin{equation}\label{EquivExponential}
\|U\|_\mathcal{H}^2\leq C\|F\|_\mathcal{H}\|F\|_\mathcal{H}\quad\text{for}\quad 0\leq \theta\leq 1.
\end{equation}

Now, notice that:
\begin{eqnarray*}
  \langle A^2v, A^\sigma w\rangle &=& \langle A^2v, A^\sigma(i\lambda u-f_1) \rangle=-i\lambda \langle A^\sigma v, A^2 u\rangle-\langle A^{1+\sigma}v, A f_1\rangle\\
  \langle A^2u, A^\sigma z\rangle &=& \langle A^2u, A^\sigma(i\lambda v-f_2) \rangle=-i\lambda \langle A^2u, A^\sigma v\rangle-\langle A^{1+\sigma}u, A f_2\rangle.
\end{eqnarray*}
Summing up, both equations and taking the real part, we have
\begin{equation}
 \label{eq00} 
 \text{Re}\{\langle A^2v, A^\sigma w\rangle+\langle A^2u, A^\sigma z\rangle\} 
   =-\text{Re}\{ \langle A^{1+\sigma}u, Af_2\rangle+\langle A^{1+\sigma}v, Af_1\rangle\}
\end{equation}
To get our first results, we should first demonstrate some lemmas.
\begin{lemma}\label{Lemma05}
Let $0\leq\theta\leq 1$ and  $\sigma\leq -1$. The solutions of equations \eqref{esp-10}-\eqref{esp-40},  satisfy the following equality
\begin{eqnarray*}
\dfrac{\gamma\alpha}{\beta}\|A^\frac{\sigma+2}{2}w\|^2 &=&\gamma\|A^{\frac{\sigma+1}{2}}z\|^2-\alpha
\text{Re}\{\langle A^{1+\sigma}u, Af_2\rangle+\langle A^{1+\sigma}v, Af_1\rangle \} \\
& &
+\dfrac{\delta\alpha}{\beta}
\text{Re}\dual{A^\frac{2\theta+\sigma}{2}z}{A^\frac{\sigma+2}{2}w}
-\dfrac{\alpha}{\beta}\text{Re}\langle f_4,A^{\sigma +1}w\rangle-\text{Re}\langle f_3,A^\sigma z\rangle\\
& &-\dfrac{\lambda\alpha}{\beta} \text{Im}\langle A^\frac{\sigma}{2} z,A^\frac{\sigma+2}{2}w\rangle-\lambda\text{Im}\langle A^\frac{\sigma+2}{2}w,A^\frac{\sigma-2}{2} z\rangle.
\end{eqnarray*}
\end{lemma}
\begin{proof}
Applying the product duality to equation \eqref{esp-30} with $A^{\sigma}z$ and recalling that the operator $A$ is self-adjoint,  we have
\begin{eqnarray*}\nonumber
\gamma\|A^{\frac{\sigma+1}{2}}z\|^2 &=&-\alpha \langle A^{2}u,A^\sigma z\rangle -i\lambda \langle w,A^\sigma z\rangle+\langle  f_3,A^\sigma z\rangle.
\end{eqnarray*}
Similarly, applying the product duality to equation (\ref{esp-40}) with $\dfrac{\alpha}{\beta}A^{\sigma+1}w$ and using the equation \eqref{esp-10} we obtain
\begin{eqnarray*}\nonumber
\dfrac{\gamma\alpha}{\beta}\|A^{\frac{\sigma+2}{2}}w\|^2 & =  & \alpha \langle A^2v,A^\sigma w\rangle+\dfrac{i\lambda\alpha}{\beta}\langle A^\frac{\sigma}{2}z,A^\frac{\sigma+2}{2} w\rangle
 +\dfrac{\delta\alpha}{\beta}\langle A^\frac{2\theta+\sigma}{2} z,A^\frac{\sigma+2}{2} w\rangle\\
 \nonumber
& &  -\dfrac{\alpha}{\beta}\langle f_4,A^{\sigma+1} w\rangle.
\end{eqnarray*}
Now, to get the conclusion of this Lemma it is sufficient to perform the subtraction of these last two equations, take the real part and use the identity  \eqref{eq00}. 
\end{proof}
Taking  $\sigma=-2$, in Lemma \ref{Lemma05}, we have

\begin{eqnarray}\nonumber
\dfrac{\gamma\alpha}{\beta}\|w\|^2 &=&\gamma\|A^{\frac{-1}{2}}z\|^2-\alpha
\text{Re}\{\langle A^{-1}u, Af_2\rangle+\langle A^{-1}v, Af_1\rangle \} \\
\label{Exp1000}
& &
+\dfrac{\delta\alpha}{\beta}
\text{Re}\dual{A^{\theta-1}z}{w}
-\dfrac{\alpha}{\beta}\text{Re}\langle f_4,A^{-1}w\rangle-\text{Re}\langle f_3,A^{-2}z\rangle\\
\nonumber
& &-\dfrac{\alpha}{\beta} \text{Im}\langle  z,A^{-1}\lambda w\rangle-\lambda\text{Im}\langle w,A^{-2} z\rangle,
\end{eqnarray}
From equation \eqref{esp-30}, we have $ A^{-1}\lambda w=i\alpha Au+i\gamma z-iA^{-1}f_3$, therefore
\begin{eqnarray}\nonumber
\hspace*{-0.5cm}-\dfrac{\alpha}{\beta} \text{Im}\langle  z,A^{-1}\lambda w\rangle & = &
-\dfrac{\alpha}{\beta} \text{Im} \langle z,i\alpha Au+i\gamma z-iA^{-1}f_3\rangle\\
\nonumber
& =& \dfrac{\alpha^2}{\beta}\text{Re}\langle A^\frac{\theta}{2}z, A^\frac{2-\theta}{2}u\rangle +\dfrac{\alpha\gamma}{\beta} \|z\|^2-\dfrac{\alpha}{\beta}\text{Re}\langle z,A^{-1}f_3\rangle\\
\label{Exp1001}
&\leq&  \dfrac{\alpha^2}{\beta}\text{Re}\langle A^\frac{\theta}{2}z, A^\frac{2-\theta}{2}u\rangle -\dfrac{\alpha}{\beta}\text{Re}\langle z,A^{-1}f_3\rangle +C\|F\|_\mathcal{H}\|F\|_\mathcal{H}.
\end{eqnarray}

Substituting  \eqref{Exp1001} into \eqref{Exp1000} and from $-\frac{1}{2}<\frac{\theta}{2}$, using \eqref{dis-10}, we have

\begin{eqnarray}
\nonumber
\dfrac{\gamma\alpha}{\beta}\|w\|^2 &\leq&C\|F\|_\mathcal{H}\|U\|_\mathcal{H}-\alpha
\text{Re}\{\langle A^{-1}u, Af_2\rangle+\langle A^{-1}v, Af_1\rangle \} \\
\label{Exp1002}
& &
+\dfrac{\delta\alpha}{\beta}
\text{Re}\dual{A^{\theta-1}z}{w}
-\dfrac{\alpha}{\beta}\text{Re}\langle f_4,A^{-1}w\rangle-\text{Re}\langle f_3,A^{-2}z\rangle\\
\nonumber
& & +\dfrac{\alpha^2}{\beta}\text{Re}\langle A^\frac{\theta}{2}z, A^\frac{2-\theta}{2}u\rangle -\dfrac{\alpha}{\beta}\text{Re}\langle z,A^{-1}f_3\rangle-\text{Im}\langle A^{-2}\lambda w, z\rangle.
\end{eqnarray}

On the other hand of the equation  \eqref{esp-30}, we have $A^{-2}\lambda w=i\alpha u+i\gamma A^{-1}z-iA^{-2}f_3$, therefore
\begin{eqnarray}\label{Exp1003}
 \text{Im}\langle A^{-2}\lambda w, z \rangle & = &
 \text{Im} \langle i\alpha u+i\gamma A^{-1}z-iA^{-2}f_3, z\rangle\\
\nonumber
& =& \alpha\text{Re}\langle A^\frac{-\theta}{2}u, A^\frac{\theta}{2}z\rangle +\gamma\|A^\frac{-1}{2}z\|^2-\text{Re}\langle A^{-2}f_3,z\rangle.
\end{eqnarray}
Now, substituting  \eqref{Exp1003} into \eqref{Exp1002}, we have

\begin{eqnarray}
\nonumber
\dfrac{\gamma\alpha}{\beta}\|w\|^2 &\leq&C\|F\|_\mathcal{H}\|U\|_\mathcal{H}-\alpha
\text{Re}\{\langle A^{-1}u, Af_2\rangle+\langle A^{-1}v, Af_1\rangle \} \\
\label{Exp1004}
& &
+\dfrac{\delta\alpha}{\beta}
\text{Re}\dual{A^{\theta-1}z}{w}
-\dfrac{\alpha}{\beta}\text{Re}\langle f_4,A^{-1}w\rangle-\text{Re}\langle f_3,A^{-2}z\rangle\\
\nonumber
& & +\dfrac{\alpha^2}{\beta}\text{Re}\langle A^\frac{\theta}{2}z, A^\frac{2-\theta}{2}u\rangle -\dfrac{\alpha}{\beta}\text{Re}\langle z,A^{-1}f_3\rangle
-\alpha\text{Re}\langle A^\frac{-\theta}{2}u, A^\frac{\theta}{2}z\rangle \\
\nonumber
& & +\text{Re}\langle A^{-2}f_3,z\rangle.
\end{eqnarray}

Applying Cauchy-Schwarz and Young inequalities, taking into account the continuous embedding $D(A^{\theta_2}) \hookrightarrow D(A^{\theta_1}),\;\theta_2>\theta_1$, $\theta-1\leq \dfrac{\theta}{2}$ and using estimative \eqref{dis-10} we have,   for  $\varepsilon>0$, there exist $k_\varepsilon>0$, such that

\begin{equation}
\label{Exp1005}
\|w\|^2 \leq C\{ \|F\|_\mathcal{H} \|U\|_\mathcal{H} \} +\varepsilon\|w\|^2+\varepsilon \|A^\frac{2-\theta}{2}u\|^2+\varepsilon\|A^\frac{-\theta}{2}u\|^2.
\end{equation}

On the other hand, by effecting the product duality of \eqref{esp-30} by $A^{-\theta}u$, we have
\begin{eqnarray*}
\alpha\|A^\frac{2-\theta}{2}u\|^2 & = & \langle w, A^{-\theta}(i\lambda u)\rangle-
\gamma\langle A^\frac{-\theta}{2}z,A^\frac{2-\theta}{2}u\rangle+\langle f_3,A^{-\theta}u\rangle\\
& = & \|A^\frac{-\theta}{2}w\|^2+\langle w, A^{-\theta}f_1 \rangle-
\gamma\langle A^\frac{-\theta}{2}z,A^\frac{2-\theta}{2}u\rangle+\langle f_3,A^{-\theta}u\rangle.
\end{eqnarray*}

Taking the real part and applying Cauchy-Schwarz and Young inequalities, taking into account the continuous embedding, $-\frac{\theta}{2}\leq\frac{\theta}{2}$, we have
\begin{equation}\label{Exp1006}
	\|A^\frac{2-\theta}{2}u\|^2 
	 \leq  C\{  \|F\|_\mathcal{H}\|U\|_\mathcal{H} \}+\|A^\frac{-\theta}{2}w\|^2.
\end{equation}
 Substituting  \eqref{Exp1006} into \eqref{Exp1005} and  taking into account the continuous embedding, $-\frac{\theta}{2}\leq\frac{2-\theta}{2}$ and $\frac{-\theta}{2}\leq 0$, we have
\begin{equation}\label{Exp1007}
\|w\|^2 \leq C\{ \|F\|_\mathcal{H} \|U\|_\mathcal{H}  \} \qquad \text{for}\qquad 0\leq\theta\leq 1.
\end{equation}


Taking the duality product between  equation (\ref{esp-30}) and $u$ and using the equation (\ref{esp-10}),  we obtain
\begin{eqnarray}\label{Exp1010}
\alpha\|Au\|^2 &= &-\gamma\dual{z}{Au}+\|w\|^2
+\langle w, f_1\rangle +\langle f_3, u\rangle.
\end{eqnarray}
Applying Cauchy-Schwarz and Young inequalities, taking into account the continuous embedding $D(A^{\theta_2}) \hookrightarrow D(A^{\theta_1}),\;\theta_2>\theta_1$,  $\frac{-1}{2}<\frac{\theta}{2}$, $0\leq \frac{\theta}{2}$ and using estimates \eqref{dis-10}  and \eqref{Exp1007} we have,   for  $\varepsilon>0$,  there exist $k_\varepsilon>0$,  such that
\begin{equation}\label{Exp1011}
\alpha\|Au\|^2\leq C\{   \|F\|_\mathcal{H}\|U\|_\mathcal{H} \}\qquad \text{for}\qquad 0\leq\theta\leq 1.	
\end{equation}

Similarly,  applying the duality product to equation (\ref{esp-40}) with $v$ and using the equation (\ref{esp-20}), we have
\begin{eqnarray}
\label{Exp1012}
\beta\|A^\frac{1}{2}v\|^2 &= &\gamma\dual{Aw}{v}+\|z\|^2
-\delta \dual{A^\frac{\theta}{2}z}{A^\frac{\theta}{2}v}+\langle z, f_2\rangle+\langle f_4, v\rangle.
\end{eqnarray}
Subtracting  \eqref{Exp1012} from \eqref{Exp1010} and taking the real part,  we have  
\begin{eqnarray*}
\beta\|A^\frac{1}{2}v\|^2 & = &\alpha \|Au\|^2 + \gamma\text{Re}\{\langle i\lambda Av-Af_2, u\rangle+\langle i\lambda Au- Af_1,v\rangle \}-\|w\|^2\\
& &-\delta\text{Re} \dual{A^\frac{\theta}{2}z}{A^\frac{\theta}{2}v}+\text{Re}\langle z, f_2\rangle+\text{Re}\langle f_4, v\rangle-\text{Re}\langle w,f_1\rangle-\text{Re}\langle f_3,u\rangle\\
&\leq&\alpha \|Au\|^2 + \gamma\lambda\text{Im}
\{\langle Av, u\rangle+\langle u, Av\rangle \}-\gamma\text{Re}\{ \langle f_2,A u\rangle+\langle Af_1,v\rangle   \}\\
& &-\delta\text{Re} \dual{A^\frac{\theta}{2}z}{A^\frac{\theta}{2}v}+\text{Re}\langle z, f_2\rangle+\text{Re}\langle f_4, v\rangle-\text{Re}\langle w,f_1\rangle-\text{Re}\langle f_3,u\rangle
\end{eqnarray*}
Now, as $\text{Im}
\{\langle Av, u\rangle+\langle u, Av\rangle \}=0$ and $\frac{\theta}{2}\leq\frac{1}{2}$,     using the estimative \eqref{Exp1011} and applying  Cauchy-Schwarz inequality and Young inequality and continuous embedding  we have the inequality
\begin{equation}\label{Exp1016}
\beta\|A^\frac{1}{2}v\|^2 \leq C\{ \|F\|_\mathcal{H}\|U\|_\mathcal{H}\}\qquad\hbox{for}\qquad 0\leq\theta\leq 1.
\end{equation}

Therefore, estimates \eqref{dis-10C}, \eqref{Exp1007}, \eqref{Exp1011}  and  \eqref{Exp1016}, condition \eqref{Exponential} the Theorem \ref{LiuZExponential} is verified for  $0\leq \theta\leq 1$.


Now let's show condition \eqref{EImaginario} the Theorem \ref{LiuZExponential}.   It'is  prove that $i\R\subset\rho(\mathbb{B})$ by contradiction, then we suppose that $i\R\not\subset \rho(\mathbb{B})$. As $0\in\rho(\mathbb{B})$ and  $\rho(\mathbb{B})$ is open, we consider the highest positive number $\lambda_0$ such that the interval  $]-i\lambda_0,i\lambda_0[\subset\rho(\mathbb{B})$ then $i\lambda_0$ or $-i\lambda_0$ is an element of the spectrum $\sigma(\mathbb{B})$. We Suppose $i\lambda_0\in \sigma(\mathbb{B})$ (if $-i\lambda_0\in \sigma(\mathbb{B})$ the proceeding is similar). Then, for $0<\delta<\lambda_0$ there exist a sequence of real numbers $(\lambda_n)$, with $\delta\leq\lambda_n<\lambda_0$, $\lambda_n\to \lambda_0$, and a vector sequence  $U_n=(u_n,v_n,w_n,z_n)\in \mathcal{D}(\mathbb{B})$ with  unitary norms, such that
\begin{eqnarray*}
\|(i\lambda_n-\mathbb{B}) U_n\|_\mathcal{H}=\|F_n\|_\mathcal{H}\to 0,
\end{eqnarray*}
as $n\to \infty$. From \eqref{Exp1011} and \eqref{Exp1016} for $0\leq\theta\leq 1$, we have 
\begin{eqnarray*}
\alpha \|A u_n\|^2 &\leq& C\{\|F_n\|_\mathcal{H}\|U_n\|_\mathcal{H}+\|F_n\|^2_\mathcal{H}\},\\
\beta\|A^{1/2}v_n\|^2 &\leq& C\{\|F_n\|_\mathcal{H}\|U_n\|_\mathcal{H}+\|F_n\|^2_\mathcal{H}\}.
\end{eqnarray*}
In addition to the estimates  and \eqref{dis-10C} and \eqref{Exp1007} for $0\leq\theta\leq 1$, we have
\begin{eqnarray*}
\|w_n\|^2 +\|z_n\|^2 \to 0.
\end{eqnarray*}

Consequently,
\begin{eqnarray*}
\alpha\|Au_n\|^2 +\beta\|A^{1/2}v_n\|^2+\|w_n\|^2+\|z_n\|^2 \to 0.
\end{eqnarray*}
Therefore, we have  $\|U_n\|_\mathcal{H}\to  0$ but this is absurd, since $\|U_n\|_\mathcal{H}=1$ for all $n\in\N$. Thus, $i\R\subset \rho(\mathbb{B})$.

 This completes the proof of condition \eqref{EImaginario} of the Theorem \ref{LiuZExponential}.


\section{ $S(t)=e^{\mathbb{B}t}$ is not analytic for $\theta \in [0,1[$ and it  is analytical for $\theta = 1 $ }

This section is divided into two subsections: In the first subsection \eqref{SSNanalitica01} we show the lack of analyticity for $0\leq\theta<1 $ and in subsection \eqref{SSNanalitica02} we test the analyticity of $S(t)$ for $\theta= 1$.         


\subsection{Lack of analyticity of $S(t)$  for $\theta\in [0,1[$}
\label{SSNanalitica01}
Semigroups $S(t)=e^{t\mathbb{B}}$ generated by $\mathbb{B}$ is not analytic when $0\leq\theta<1$.
\begin{theorem}
\label{lack KV and structural} Let $S(t)=e^{t\mathbb{B}}$ be the $C_{0}%
$-semigroups of contractions over the Hilbert space $\mathcal{H}$ associated
with the system \eqref{ISplacas0}--\eqref{ISplacas4}  is not analytic when $\theta\in\left[  0,1\right[$.
\end{theorem}
\begin{proof}

Now we show that the corresponding semigroups is not analytic for  $0\leq\theta<1$.  Let us construct a sequence $F_n$ such that the solutions of
\begin{equation*}
i\lambda_n U_n-\mathbb{B}U_n=F_n.
\end{equation*}
satisfies $|\lambda_n|\|U_n\|_\mathcal{H}\to\infty$, which in particular implies
\begin{equation*}
\|\lambda_n(i\lambda_n I-\mathbb{B})^{-1}F_n\|_\mathcal{H}\to\infty
\end{equation*}
which means that the corresponding semigroups is not analytic. 

	The  spectrum of  operator $A=-\Delta$ defined in \eqref{Omenoslaplaciano}  is constituted by positive eigenvalues $(\sigma_n)$ such that $\sigma_n\to \infty$ as $n\to  \infty$. For $n\in \N$ we denote with   $e_n$  an unitary $L^2$-norm eigenvector associated to the eigenvalue $\sigma_n$, that is:
	\begin{equation}\label{auto-10case03}
	Ae_n=\sigma_ne_n,\quad A^\theta e_n=\sigma_n^\theta e_n,\quad  \|e_n\|_{L^2(\Omega)}=1,\quad\text{for}\quad 0\leq \theta<1 ,\; n\in\N
	\end{equation}
Let's show that the right side of inequality \eqref{Analiticity} for $\theta\in[0,1)$ is not verified. Consider the eigenvalues and eigenvectors of the operator $A$ as in \eqref{Omenoslaplaciano} and \eqref{auto-10case03} respectively.
	
	Let $F_n=(0,0,-e_n,0)\in \mathcal{H}$.
	The solution $U=(u_n,v_n,w_n,z_n)$ of the system $(i\lambda I-\mathbb{B})U_n=F_n$ satisfies $w_n=i\lambda u_n$, $z=i\lambda v_n$ and the following equations
	\begin{eqnarray*}
		\lambda^2  u_n-\alpha A^2u_n-i\lambda\gamma Av_n&=& e_n,\\
		\lambda^2  v_n-\beta Av_n+i\gamma\lambda Au_n
		-i\lambda\delta A^{\theta} v_n&=&0.
	\end{eqnarray*}
	Let us see whether this system admits solutions of the form
	\begin{equation*}
	u_n=\mu_n e_n,\quad v_n=\nu_n e_n,
	\end{equation*}
	for some complex numbers $\mu_n$ and $\nu_n$. Then, the numbers $\mu_n$, $\nu_n$ should satisfy the algebraic system
	\begin{eqnarray}\label{eq001system}
	\big\{\lambda^2_n- \alpha\sigma_n^2\big\}\mu_n-i\lambda_n\gamma\sigma_n\nu_n &=& 1,\\
	\label{eq002system}
	i\lambda_n\gamma\sigma_n\mu_n+\big\{\lambda^2_n
	-\beta\sigma_n-i\delta\sigma_n^{\theta}\lambda_n\big\}\nu_n&=& 0.
	\end{eqnarray}
	
	On the other hand solving the system \eqref{eq001system}-\eqref{eq002system}, we find that
	\begin{eqnarray}\label{Muexpre01}
	\mu_n=\frac{\big\{p_{2,n}(\lambda^2_n)
		-i\delta\sigma_n^{\theta}\lambda_n\big\}}
	{p_{1,n}(\lambda^2_n)p_{2,n}(\lambda^2_n)-\gamma^2\lambda^2_n\sigma_n^2
		-i\delta\sigma_n^{\theta}\lambda_n p_{1,n}(\lambda^2_n)},
	\end{eqnarray}
	where
	\begin{eqnarray}\label{P1P2}
	p_{1,n}(\lambda^2_n):=\lambda^2_n-\alpha\sigma_n^2\qquad\text{and}\qquad p_{2,n}(\lambda^2_n)=\lambda_n^2-\beta\sigma_n.
	\end{eqnarray}

	Taking $s_n=\lambda_n^2$ and considering  the polynomial
	\begin{eqnarray*}
	q_n(s_n) & := & p_{1,n}(s_n)p_{2,n}(s_n)-\gamma^2\sigma_n^2s_n\\
	&=& s_n^2-[(\alpha+\gamma^2)\sigma_n^2+\beta\sigma_n]s_n+\alpha\beta\sigma_n^3.	
	\end{eqnarray*}
	Now,  taking   $q_n(s_n)=0$,  we have
	the roots of the polynomial $q_n(s_n)$ are given by
	\begin{eqnarray}
	\label{Sn+-}
	s_n^{\pm}& = & \dfrac{[(\alpha+\gamma^2)\sigma_n^2+\beta\sigma_n]  \pm \sigma_n\sqrt{(\alpha+\gamma^2)^2\sigma_n^2+2\beta(\gamma^2-\alpha)\sigma_n+\beta^2}}{2}.
	\end{eqnarray}	
		
		Thus, if we introduce the notation $x_n\approx y_n$ meaning that $\displaystyle\lim_{n\to \infty}\frac{|x_n|}{|y_n|}$ is a positive real number.
			
	Taking $s_n=s_n^+$ from  equation \eqref{Sn+-}, we have 
\begin{equation}\label{Estsnelambdan}
s_n\approx \sigma_n^2\qquad \text{and}\qquad \lambda_n\approx \sigma_n.
\end{equation}	
Then
	\begin{equation}\label{pnMAIS}
	p_{2,n}(s_n)=s_n-\beta\sigma_n \approx \sigma_n^2.
	\end{equation}
		From $q_n(s_n)=0$ in \eqref{Muexpre01}, we have
	\begin{equation}
	\label{eq001mu}
	\mu_n  =  \frac{\big\{p_{2,n}(\lambda_n^2)
		-i\delta\lambda_n\sigma_n^{\theta}\big\}}{ -i\delta\sigma_n^{\theta}\lambda_np_{1,n}(\lambda^2)}
	=  \dfrac{p_{2,n}
		(\lambda_n^2)}{\gamma^2\lambda_n^2\sigma_n^2}
	+i\dfrac{p^2_{2,n}(\lambda_n^2)}
	{\delta\gamma^2\lambda_n^3\sigma_n^{2+\theta}}.
	\end{equation}
	Therefore
	\begin{equation}\label{EstimatDMU0a1}
	|\mu_n|\approx |\lambda_n|^{-1-\theta}.
	\end{equation}
Finally, of \eqref{auto-10case03}  for $C>0$, the solution $U_n$ of the system $(i\lambda_n-\mathbb{B})U=F_n$, satisfies
\begin{equation}\label{Estima001NA}
\|U_n\|_\mathcal{H}\geq C\|w_n\|=C|\lambda_n|\|u_n\|=
C|\lambda_n||\mu_n|\|e_n\|=C|\lambda_n||\mu_n|=C|\lambda_n|^{-\theta}\quad \text{for}\quad 0\leq\theta<1.
\end{equation}

Then,  using estimates \eqref{EstimatDMU0a1}
in \eqref{Estima001NA}, for $\delta>0$   and  $0\leq\theta<1$,   we obtain 
\begin{equation}
\label{EstimaU0a1e1a2}
|\lambda_n|\|U_n\|_\mathcal{H}\geq \delta|\lambda_n|^{1-\theta}  \qquad  \Longrightarrow\qquad |\lambda_n\|U_n\|_\mathcal{H}\to \infty.
\end{equation}
From where our conclusion follows.
\end{proof}

\subsection{Analyticity  of  $S(t)$ for $\theta=1$} 
\label{SSNanalitica02}

In this subsection we show the analyticity the $S(t)$ for $\theta=1$  using 
Theorem \ref{LiuZAnaliticity}, specifically checking to condition \eqref{Analiticity}( $|\lambda|\|(i\lambda I-\mathbb{B})^{-1}F\|^2_\mathcal{H} \leq C_\delta\{\|F\|_\mathcal{H} \|U\|_\mathcal{H}\}$)
\begin{remark}\label{Lemma06} 
Let $\delta>0$. Exist $C_\delta>0$ such that,	for $0\leq \theta\leq 1$, we have $\frac{\theta-1}{2}\leq 0$. Applying continuous immersions and inequality \eqref{Exp1007}, we have 
	\begin{eqnarray*}
		\|A^{\frac{\theta-1}{2}}w\|^2&\leq&  C_\delta\{ \|F\|_\mathcal{H} \|U\|_\mathcal{H}  \} \qquad \text{for}\qquad 0\leq\theta\leq 1.
	\end{eqnarray*}
\end{remark}
\begin{lemma}\label{Lemma07}
	Let $\delta>0$. Exist $C_\delta>0$ such that the  solutions of equations
	(\ref{esp-10})-(\ref{esp-40}) for $|\lambda|\geq \delta$, satisfy
	\begin{eqnarray*}
	\|A^\frac{\theta}{2}w \|^2& \leq &C_\delta\{ \|F\|_\mathcal{H} \|U\|_\mathcal{H}  \} \qquad \text{for}\qquad 0\leq\theta\leq 1.
	\end{eqnarray*}
\end{lemma}
\begin{proof}
	From $0\leq\theta\leq 1$, then $\sigma=\theta-2\leq -1$. Therefore taking $\sigma=\theta-2$ in the Lemma \ref{Lemma05}, we have
	\begin{eqnarray}
	\nonumber
\hspace*{-1.1cm}	\dfrac{\gamma\alpha}{\beta}\|A^\frac{\theta}{2}w\|^2 &=&\gamma\|A^{\frac{\theta-1}{2}}z\|^2-\alpha
	\text{Re}\{\langle A^{\theta-\frac{1}{2}}u, A^\frac{1}{2}f_2\rangle+\langle A^{\theta-1}v, Af_1\rangle \} \\
	\label{Eq00Lemma07}
	& &
	+\dfrac{\delta\alpha}{\beta}
	\text{Re}\dual{A^\frac{3\theta-2}{2}z}{A^\frac{\theta}{2}w}
	-\dfrac{\alpha}{\beta}\text{Re}\langle f_4,A^{\theta-1}w\rangle-\text{Re}\langle f_3,A^{\theta-2} z\rangle\\
	\nonumber
	& &-\dfrac{\alpha}{\beta} \text{Im}\langle z,\lambda A^{\theta-1}w\rangle-\lambda\text{Im}\langle A^\frac{\theta}{2}w,A^\frac{\theta-4}{2} z\rangle.
	\end{eqnarray}
	From equation \eqref{esp-30}, we have $\lambda A^{\theta-1} w=i\alpha A^{\theta+1}u+i\gamma A^\theta z-iA^{\theta-1}f_3$,  therefore
	\begin{eqnarray}
\hspace*{-.8cm}	-\dfrac{\alpha}{\beta}\text{Im}\dual{z}{\lambda A^{\theta-1}w} & =& -\dfrac{\alpha}{\beta}\text{Im}\dual{A^\frac{\theta}{2}z}{i\alpha A^\frac{\theta+2}{2}u}+\dfrac{\alpha\gamma}{\beta}\|A^\frac{\theta}{2}z\|-\dfrac{\alpha}{\beta}\text{Re}\dual{z}{A^{\theta-1}f_3}
	\label{Eq01Lemma07}
\end{eqnarray}	
	Applying Cauchy-Schwarz and Young inequalities, estimative \eqref{dis-10} and  for $1>\varepsilon>0$, exist $K_\varepsilon>0$, we get
		\begin{eqnarray}
		\bigg|\dfrac{\alpha}{\beta}\text{Im}\dual{z}{\lambda A^{\theta-1}w}\bigg |& \leq& K_\varepsilon  \|F\|_\mathcal{H} \|U\|_\mathcal{H} +\varepsilon \|A^\frac{2+\theta}{2}u\|^2
	\label{Eq01Lemma08}
	\end{eqnarray}	
	On the outer hand,   applying the product duality to equation \eqref{esp-30} with $A^\theta u$ and recalling that the operator $A$ is seft-adjoint,  we obtain
	\begin{eqnarray*}
	\alpha	\|A^\frac{2+\theta}{2}u\|^2 &= &\dual{w}{A^\theta(i\lambda u)}-\gamma\dual{A^\frac{\theta}{2}z} {A^\frac{2+\theta}{2}u}+\dual{f_3}{A^\theta u}\\
	&=& \|A^\frac{\theta}{2}w\|^2+\dual{w}{A^\theta f_1}-\gamma\dual{A^\frac{\theta}{2}z}{A^\frac{2+\theta}{2}u}+\dual{f_3}{A^\theta},
		\end{eqnarray*}
	now applying Cauchy-Schwarz and Young inequalities for every $\varepsilon>0$, there exists a positive constant $K_\varepsilon$, independent of $\lambda$, such that
	\begin{equation}\label{Eq02Lemma07}
	\|A^\frac{2+\theta}{2}u\|^2 \leq C\{ \|F\|_\mathcal{H} \|U\|_\mathcal{H} \}+\|A^\frac{\theta}{2} w\|^2.
	\end{equation}
	Using \eqref{Eq02Lemma07} in  \eqref{Eq01Lemma08},  we obtain
	\begin{equation}\label{Eq03Lemma07}
	-\dfrac{\lambda\alpha}{\beta}\text{Im}\dual{z}{A^{\theta-1}w} \leq \varepsilon\|A^\frac{\theta}{2}w\|^2+ C\{ \|F\|_\mathcal{H} \|U\|_\mathcal{H} \}.
	\end{equation}
	Similarly   as $-\lambda{\rm Im}\dual{A^\frac{\theta}{2}w}{A^\frac{\theta-4}{2}z}= -\lambda{\rm Im}\dual{A^{\theta-2}w}{z}$ and from equation \eqref{esp-30}, we have $A^{\theta-2}\lambda w=i\alpha A^\theta  u+i\gamma A^{\theta-1} z-iA^{\theta-2}f_3$,  therefore,
	\begin{eqnarray}
	\hspace*{-.8cm}	-\lambda\text{Im}\dual{A^{\theta-2}w}{z} & =& {\rm Im}\{ -i\alpha\dual{A^\theta u}{z}-i\gamma\|A^\frac{\theta-1}{2}z\|^2+i\dual{A^{\theta-2}f_3}{z}\}
	\label{Eq04Lemma07}
	\end{eqnarray}	
	Applying Cauchy-Schwarz and Young inequalities, estimative \eqref{dis-10} and  for $1>\varepsilon>0$, exist $K_\varepsilon>0$, we get
	\begin{eqnarray}
	\hspace*{-.8cm}	-\text{Im}\dual{A^{\theta-2}w}{z} & \leq& K_\varepsilon\|A^\frac{\theta}{2}z\|^2+   \varepsilon \|A^\frac{\theta}{2}u\|^2 +C\|A^\frac{\theta-1}{2}z\|^2+C \|F\|_\mathcal{H} \|U\|_\mathcal{H}.
	\label{Eq05Lemma07}
	\end{eqnarray}	
From $\frac{\theta-1}{2}<\frac{\theta}{2}\leq\frac{2+\theta}{2}$ using continuous embedding and estimates \eqref{dis-10} and \eqref{Eq02Lemma07}, we obtain

\begin{equation}\label{Eq06Lemma07}
-\text{Im}\dual{A^{\theta-2}w}{z} \leq  \varepsilon\|A^\frac{\theta}{2}w\|^2+ C_\delta\{ \|F\|_\mathcal{H} \|U\|_\mathcal{H} \}.
\end{equation}

	Applying Cauchy-Schwarz and Young inequalities in equation \eqref{Eq00Lemma07},  for $1>\varepsilon>0$, exist $K_\varepsilon>0$ and estimates \eqref{Eq03Lemma07} and \eqref{Eq06Lemma07} and from $\frac{\theta-4}{2}<\frac{\theta-1}{2}<\frac{\theta}{2}$ using continuous embedding   for every $\varepsilon>0$, there exists a positive constant $K_\varepsilon$, independent of $\lambda$, such that
	\begin{eqnarray}\label{Eq7050}
	\|A^\frac{\theta}{2}w\|^2 &\leq & C\|A^\frac{\theta}{2}z\|^2
	+ C_\delta\{ \|F\|_\mathcal{H} \|U\|_\mathcal{H}  \}
	+\varepsilon\|A^\frac{\theta}{2}w\|^2.
	\end{eqnarray}
	Finally from inequality \eqref{dis-10}  in the inequality \eqref{Eq7050} finish to proof.
	
\end{proof}
\begin{remark}\label{Remark01}
	Using Lemma \ref{Lemma07} in the inequality \eqref{Eq02Lemma07}, we have
	\begin{eqnarray}
	\label{Lemma09A} 
	\|A^\frac{\theta+2}{2}u\|^2&\leq& C_\delta\{ \|F\|_\mathcal{H} \|U\|_\mathcal{H}\}\qquad{\rm for}\qquad 0\leq \theta\leq 1.
	\end{eqnarray}
	And  taking $\theta=1$ in Lemma \ref{Lemma07}, we have 
	\begin{equation}\label{Estimativa02}
	\|A^\frac{1}{2}w\|\leq C_\delta\{ \|F\|_\mathcal{H} \|U\|_\mathcal{H} \}^\frac{1}{2}.
	\end{equation}
\end{remark}


\begin{remark}\label{Remark02}
	Taking $\theta=1$ in inequality \eqref{Lemma09A}  to Remark \ref{Remark01}, we have 
	\begin{equation}\label{Estimativa03}
	\|A^\frac{3}{2}u\|^2\leq C_\delta\{ \|F\|_\mathcal{H} \|U\|_\mathcal{H} \}.
	\end{equation}
\end{remark}
\begin{lemma}\label{Lemma12} Let $\theta=1 $ and $\delta>0$. Exist $C_\delta>0$ such that the  solutions of equations (\ref{esp-10})-(\ref{esp-40})  for $|\lambda|\geq \delta$, satisfy:
	\begin{eqnarray*}
|\lambda| \|z\|^2&\leq& C_\delta\{\|F\|_\mathcal{H}\|U\|_\mathcal{H}\}.
	\end{eqnarray*}
\end{lemma}
\begin{proof} 
Applying the product duality to equation \eqref{esp-40}  with  $z$ and recalling that the operator $A$ is self-adjoint,  we have
\begin{eqnarray}\nonumber
i\lambda\|z\|^2=-\beta\dual{A^\frac{1}{2}v}{A^\frac{1}{2}z}+\gamma\dual{A^\frac{1}{2}w}{ A^\frac{1}{2}z}
-\delta \|A^\frac{\theta}{2}z\|^2+\dual{f_4}{z}.
\end{eqnarray}
Taking the imaginary part and using Cauchy-Schwarz and Young  inequalities,  we obtain
\begin{equation}\label{Eq01Lemma12}
|\lambda|\|z\|\leq C_\delta\{ \|A^\frac{1}{2}z\|^2+\|A^\frac{1}{2}w\|^2+\|A^\frac{1}{2}v\|^2+\|f_4\|\|z\|\}.
\end{equation}
From estimates \eqref{dis-10}, \eqref{Exp1016},   \eqref{Estimativa02} and norms  $\|F\|_\mathcal{H}^2$ and $\|U\|^2_\mathcal{H}$, finish to proof.

\end{proof}

\begin{lemma}\label{Lemma14}
	Let  $\theta=1$ and $\delta>0$,  exists $C_\delta>0$,  such that,  the solutions of equations (\ref{esp-10})-(\ref{esp-40}) satisfy the following inequality:
	\begin{eqnarray*}
		| \lambda|\|w\|^2 \leq  C_\delta\{ \|F\|_\mathcal{H}\|U\|_\mathcal{H}\}.
	\end{eqnarray*}
\end{lemma}
\begin{proof}
Considering $\theta=1$,  applying the product duality to equation \eqref{esp-30} with $w$ and recalling that the operator $A$  is
self-adjoint, we have
\begin{eqnarray}
i\lambda\|w\|^2=-\alpha\dual{A^\frac{3}{2}u}{A^\frac{1}{2}w}-\gamma\dual{A^\frac{1}{2}z}{A^\frac{1}{2}w}+\dual{f_3}{w}.
\end{eqnarray}
Taking the imaginary part and using Cauchy-Schwarz and Young  inequalities,  we obtain
\begin{equation}\label{Eq01Lemma14}
|\lambda|\|w\|^2\leq C_\delta\{\|A^\frac{3}{2}u\|^2+ \|A^\frac{1}{2}z\|^2+\|A^\frac{1}{2}w\|^2+\|f_3\|\|w\|\}.
\end{equation}
From estimates \eqref{dis-10}, \eqref{Estimativa02}, \eqref{Estimativa03},    and norms  $\|F\|_\mathcal{H}^2$ and $\|U\|^2_\mathcal{H}$, finish to proof.
\end{proof}
\begin{lemma}\label{Lemma15}
	Let  $\theta=1$ and $\delta>0$,  exists $C_\delta>0$,  such that,  the solutions of equations (\ref{esp-10})-(\ref{esp-40}) satisfy the following inequality:
	\begin{eqnarray}\label{EAnaliti010}
		| \lambda|\|Au\|^2 \leq  C_\delta\{ \|F\|_\mathcal{H}\|U\|_\mathcal{H}\}.
	\end{eqnarray}
\end{lemma}
\begin{proof}
Considering $\theta=1$,  applying the product duality to equation \eqref{esp-30} with $w$, using \eqref{esp-10} and recalling that the operator $A$  is
self-adjoint, we have
\begin{eqnarray}
i\lambda\|w\|^2=-\alpha\dual{A^2 u}{i\lambda u-f_1}-\gamma\dual{A^\frac{1}{2}z}{A^\frac{1}{2}w}+\dual{f_3}{w}.
\end{eqnarray}
Equivalent
\begin{eqnarray*}
\alpha i\lambda\|Au\|^2=-i\lambda\|w\|^2+\dual{Au}{Af_1}-\gamma\dual{A^\frac{1}{2}z}{A^\frac{1}{2}w}+\dual{f_3}{w}.
\end{eqnarray*}
Taking the imaginary part and using Cauchy-Schwarz and Young  inequalities,  we obtain
\begin{equation}\label{Eq01Lemma15}
|\lambda|\|Au\|^2\leq C_\delta\{|\lambda|\|w\|^2+ \|A^\frac{1}{2}z\|^2+\|A^\frac{1}{2}w\|^2+\|Au\|\|Af_1\|+\|f_3\|\|w\|\}.
\end{equation}
From estimates \eqref{dis-10},  Lemma \ref{Lemma14}, \eqref{Estimativa02}   and norms  $\|F\|_\mathcal{H}^2$ and $\|U\|^2_\mathcal{H}$, finish to proof.

\end{proof}
Finally,  the following lemma estimates the term $|\lambda|\|A^\frac{1}{2}v\|^2$. 
\begin{lemma}\label{Lemma16}
	Let  $\theta=1$ and $\delta>0$,  exists $C_\delta>0$,  such that,  the solutions of equations (\ref{esp-10})-(\ref{esp-40}) satisfy the following inequality:
	\begin{eqnarray}\label{EAnaliti011}
		| \lambda|\|A^\frac{1}{2}v\|^2 \leq  C_\delta\{ \|F\|_\mathcal{H}\|U\|_\mathcal{H}\}.
	\end{eqnarray}
\end{lemma}
\begin{proof}
Considering $\theta=1$,  applying the product duality to equation \eqref{esp-40} with $z$, using \eqref{esp-10} and recalling that the operator $A$  is
self-adjoint, we have
\begin{eqnarray}
i\lambda\|z\|^2=-\beta\dual{A v}{i\lambda v-f_2}+\gamma\dual{A^\frac{1}{2}w}{A^\frac{1}{2}z}-\delta\|A^\frac{1}{2}z\|^2+\dual{f_4}{z}.
\end{eqnarray}
Equivalent
\begin{eqnarray*}
i\beta \lambda\|A^\frac{1}{2}v\|^2=-i\lambda\|z\|^2+\beta\dual{A^\frac{1}{2}v}{A^\frac{1}{2}f_2}+\gamma\dual{A^\frac{1}{2}w}{A^\frac{1}{2}z}-\delta\|A^\frac{1}{2}z\|^2+\dual{f_4}{z}.
\end{eqnarray*}
Taking the imaginary part and using Cauchy-Schwarz and Young  inequalities,  we obtain
\begin{equation}\label{Eq01Lemma16}
|\lambda|\|A^\frac{1}{2}v\|^2\leq C_\delta\{|\lambda|\|z\|^2+ \|A^\frac{1}{2}z\|^2+\|A^\frac{1}{2}w\|^2+\|A^\frac{1}{2}v\|\|A^\frac{1}{2}f_2\|+\|f_4\|\|z\|\}.
\end{equation}
From estimates \eqref{dis-10},  Lemma \ref{Lemma15}, \eqref{Estimativa02}   and norms  $\|F\|_\mathcal{H}^2$ and $\|U\|^2_\mathcal{H}$, finish to proof.

\end{proof}

For $\theta=1$,   summing estimates the  Lemmas \ref{Lemma12}, \ref{Lemma14}, \ref{Lemma15}  and Lemma  \ref{Lemma16},  we have
\begin{equation}\label{EAnaliti005}
|\lambda|\|U\|^2_\mathcal{H}\leq C_\delta\|F\|_\mathcal{H}\|U\|_\mathcal{H} \quad\Longrightarrow\quad \|\lambda(i\lambda I-\mathbb{B})^{-1}\|_{\mathcal{L}({\mathcal{H}})}\leq C_\delta. 
\end{equation}
Therefore for $\theta=1$,  the condition  \eqref{Analiticity}  is also verified,  so the proof of the Theorem \ref{LiuZAnaliticity} is finished.
\section{ $S(t)=e^{\mathbb{B}t}$ is  of Gevrey sharp class $s> \frac{1}{\theta}$   when the parameter $\theta$ lies in the interval $]0,1[$.}
Before exposing our results, it is useful to recall the next definition and result  presented in \cite{SCRT1990, Tebou-2020} (adapted from
\cite{TaylorM}, Theorem 4, p. 153]).

\begin{definition}\label{Def1.1Tebou} Let $t_0\geq 0$ be a real number. A strongly continuous semigroup $S(t)$, defined on a Banach space $ \mathcal{H}$, is of Gevrey class $s > 1$ for $t > t_0$, if $S(t)$ is infinitely differentiable for $t > t_0$, and for every compact set $K \subset (t_0,\infty)$ and each $\mu > 0$, there exists a constant $ C = C(\mu, K) > 0$ such that
	\begin{equation}\label{DesigDef1.1}
	||S^{(n)}(t)||_{\mathcal{L}( \mathcal{H})} \leq  C\mu ^n(n!)^s,  \text{ for all } \quad t \in K, n = 0,1,2...
	\end{equation}
\end{definition}
\begin{theorem}[\cite{TaylorM}]\label{Theorem1.2Tebon}
	Let $S(t)$  be a strongly continuous and bounded semigroups on a Hilbert space $ \mathcal{H}$. Suppose that the infinitesimal generator $\mathbb{B}$ of the semigroups $S(t)$ satisfies the following estimate, for some $0 < \tau < 1$:
	\begin{equation}\label{Eq1.5Tebon2020}
	\lim\limits_{|\lambda|\to\infty} \sup |\lambda |^\tau ||(i\lambda I-\mathbb{B})^{-1}||_{\mathcal{L}( \mathcal{H})} < \infty. 
	\end{equation}
	Then $S(t)$  is of Gevrey  class  $s$   for $t>0$, for every   $s >\dfrac{1}{\tau}$.
\end{theorem}
\begin{lemma}\label{Lemma01Gevrey}
Let $ 0<\theta<1$ and $\delta_1>0$, exists $C_{\delta_1}>0$, such that, the solutions of equations \eqref{esp-10}--\eqref{esp-40} satisfy the following inequality 
\begin{eqnarray}
\label{Eq01Lemma01G}
(i)\quad|\lambda|\|A^\frac{\theta-1}{2}z\|^2 &\leq & C_{\delta_1} \|F\|_\mathcal{H}\|U\|_\mathcal{H}\quad{\rm for}\quad 0<\theta<1.\\
(ii)\quad |\lambda|\|A^\frac{\theta-1}{2}w\|^2 &\leq & C_{\delta_1} \|F\|_\mathcal{H}\|U\|_\mathcal{H}\quad{\rm for}\quad 0<\theta<1.
\end{eqnarray}
\end{lemma}
\begin{proof}
{\bf (i)}  Taking the duality product between equation \eqref{esp-40} and $A^{\theta-1}z$, using advantage of the seft-adjointness of the powers of the operator $A$, we get
\begin{equation*}
i\lambda\|A^\frac{\theta-1}{2}z\|^2=-\beta\dual{A^\frac{1}{2}v}{A^{\theta-\frac{1}{2}}z}+\gamma\dual{A^\frac{\theta}{2}w}{A^\frac{\theta}{2}z}-\delta\|A^\frac{2\theta-1}{2}z\|^2+\dual{f_4}{A^{\theta-1} z}.
\end{equation*}
Taking imaginary part, applying Cauchy-Schwarz inequalities, we have
\begin{equation}
\label{Eq02Lemma02G}
|\lambda|\|A^\frac{\theta-1}{2}z\|^2\leq  C_{\delta_1}\{\|A^\frac{1}{2}v\|^2+\|A^{\theta-\frac{1}{2}}z\|^2+\|A^\frac{\theta}{2}w\|^2+\|A^\frac{\theta}{2}z\|^2\}+\|f_4\|\|A^{\theta-1}z\|,
\end{equation}
As for $0<\theta<1$, we have $\theta-\frac{1}{2}\leq \frac{\theta}{2}$ and $\theta-1<0$, using continuous embedding and estimates \eqref{dis-10}, \eqref{Exp1016}  and   Lemma \ref{Lemma07}, we finish proof this item.
\\
Proof. ${\mathbf(ii)}$  Taking the duality product between equation \eqref{esp-30} and $A^{\theta-1}w$, using advantage of the seft-adjointness of the powers of the operator $A$, we get
\begin{equation*}
i\lambda\|A^\frac{\theta-1}{2}w\|^2=-\alpha\dual{A^\frac{\theta+2}{2}u}{A^\frac{\theta}{2}w}-\gamma\dual{A^\frac{1}{2}z}{A^{\theta-\frac{1}{2}}w}+\dual{f_3}{A^{\theta-1} w}.
\end{equation*}
Taking imaginary part, applying Cauchy-Schwarz inequalities, we have
\begin{equation}
\label{Eq03Lemma02G}
|\lambda|\|A^\frac{\theta-1}{2}w\|^2\leq  C_{\delta_1}\{\|A^\frac{\theta+2}{2}u\|^2+\|A^\frac{\theta}{2}w\|^2+\|A^\frac{\theta}{2}z\|^2\}+\|f_3\|\|A^{\theta-1}w\|,
\end{equation}
As for $0<\theta<1$, we have  $\theta-1<0$, using continuous embedding and estimates \eqref{dis-10}, Remark \ref{Remark01}  and  Lemma \ref{Lemma07}, we finish proof this item.
\end{proof}

Our main result in this section is as follows:
\begin{theorem} \label{TCGevrey}
Let  $S(t)=e^{\mathbb{B}t}$  strongly continuos-semigroups of contractions on the Hilbert space $ \mathcal{H}$, the semigroups $S(t)$ is of Gevrey class $s$ for every $s>\dfrac{1}{\tau}$ for $\tau\in ]0,1[$, as there exists a positive constant $C$ such that we have the resolvent estimative:
	\begin{equation}\label{Eq1.6Tebon2020}
	|\lambda |^\tau ||(i\lambda I-\mathbb{B})^{-1}||_{\mathcal{L}( \mathcal{H})} \leq C, \quad \lambda\in\R. 
	\end{equation}	
\end{theorem}
\begin{proof}
We will initially show that for $0<\theta<1$ it is verified:
\begin{equation}\label{Eq00Gevrey}
|\lambda|^\theta\|z\|^2\leq C_{\delta_1}\|F\|_\mathcal{H}\|\|U\|_\mathcal{H}\qquad{\rm and}\qquad |\lambda|^\theta\|w\|^2\leq C_{\delta_1}\|F\|_\mathcal{H}\|\|U\|_\mathcal{H}.
\end{equation}

As for $0<\theta<1$, we have $0\in \big[\frac{\theta-1}{2},\frac{\theta}{2}\big]$. We are going to use an interpolation inequality. Since
\begin{equation*}
0=\phi\bigg(\dfrac{\theta-1}{2}\bigg)+(1-\phi)\bigg(\dfrac{\theta}{2}\bigg), \quad{\rm for}\quad \phi=\theta\quad{\rm and}\quad 1-\phi=1-\theta,
\end{equation*}
using inequalities \eqref{dis-10} and item $(i)$ of Lemma \ref{Lemma01Gevrey}, we get that
\begin{eqnarray*}
\|z\|^2 &\leq & C(\|A^\frac{\theta-1}{2}z\|^2)^\theta(\|A^\frac{\theta}{2}z\|^2)^{1-\theta}\\
&\leq & C|\lambda|^{-\theta}\{\|F\|_\mathcal{H}\|U\|_\mathcal{H}\}^\theta\{\|F\|_\mathcal{H}\|U\|_\mathcal{H}\}^{1-\theta}.
\end{eqnarray*}
From where we ended the proof of \eqref{Eq00Gevrey}$_1$

On the other hand. As for $0<\theta<1$, we have $0\in \big[\frac{\theta-1}{2},\frac{\theta}{2}\big]$. We are going to use an interpolation inequality. Since
\begin{equation*}
0=\phi\bigg(\dfrac{\theta-1}{2}\bigg)+(1-\phi)\bigg(\dfrac{\theta}{2}\bigg), \quad{\rm for}\quad \phi=\theta\quad{\rm and}\quad 1-\phi=1-\theta,
\end{equation*}
using Lemma \ref{Lemma07} and item $(ii)$ of Lemma \ref{Lemma01Gevrey}, we get that
\begin{eqnarray*}
\|w\|^2 &\leq & C(\|A^\frac{\theta-1}{2}w\|^2)^\theta(\|A^\frac{\theta}{2}w\|^2)^{1-\theta}\\
&\leq & C|\lambda|^{-\theta}\{\|F\|_\mathcal{H}\|U\|_\mathcal{H}\}^\theta\{\|F\|_\mathcal{H}\|U\|_\mathcal{H}\}^{1-\theta}.
\end{eqnarray*}
From where we ended the proof of \eqref{Eq00Gevrey}$_2$
\\
Now we will estimate the term $|\lambda|\|Au\|^2$.  Making the duality product between equation \eqref{esp-30}  and $\lambda u$ and using the equation \eqref{esp-10}, we have
\begin{eqnarray*}
	\alpha_1\lambda \|Au\|^2&=&\lambda\dual{w}{i\lambda u}-\gamma\dual{\dfrac{\lambda}{|\lambda|^\frac{1}{2}}z}{|\lambda|^\frac{1}{2}Au} +\dual{f_3}{\lambda u}\\
	\nonumber
	& =& \lambda\|w\|^2+\dual{i\alpha_1A^2u+i\gamma Az-if_3}{f_1}-\gamma\dual{\dfrac{\lambda}{|\lambda|^\frac{1}{2}}z}{|\lambda|^\frac{1}{2}Au} +\dual{f_3}{-iw-i f_1} .
\end{eqnarray*}
Applying Cauchy-Schwarz and Young inequalities, for $\varepsilon>0$, there exists a positive constant $K_\varepsilon$, independent of $\lambda$,  such that:
\begin{eqnarray}\label{Eq001AnalyR1N}
|\lambda|\|Au\|^2&\leq &C|\lambda|\|w\|^2+C\{|\dual{Au}{Af_1}|+|\dual{z}{Af_1}|+|\dual{f_3}{f_1}|\}
\\
\nonumber
& & +K_\varepsilon|\lambda|\|z\|^2+\varepsilon|\lambda|\|Au\|^2+C|\dual{f_3}{w}|.
\end{eqnarray}
Now applying Cauchy-Schwarz and Young inequalities   and from estimative \eqref{Eq00Gevrey},  we have
\begin{eqnarray}\label{Eq001AnalyRN2}
|\lambda|\alpha_1\|Au\|^2&\leq &C|\lambda|^{1-\theta}\{\|F\|_\mathcal{H}\|U\|_\mathcal{H}\}    \quad\text{for}\quad 0\leq\theta\leq 1.
\end{eqnarray}

Finally we'll get the  estimative for $|\lambda|\|A^\frac{1}{2}v\|^2$, taking the duality product between equation \eqref{esp-30}  and $w$ and using the equation \eqref{esp-10}, we have
\begin{eqnarray}\label{Eq001AnalyR}
i\lambda \|w\|^2-i\lambda\alpha_ 1\|Au\|^2&=&-\gamma\langle A^\frac{1}{2}z, A^\frac{1}{2}w\rangle +\alpha_1\langle Au, A f_1\rangle+\langle f_3, w\rangle.
\end{eqnarray}
Now, taking the duality product between equation \eqref{esp-40}  and $z$ and using the equation \eqref{esp-20}, we have
\begin{equation}\label{Eq002AnalyR}
i\lambda\|z\|^2+\delta \|A^\frac{\theta}{2}z \|^2=i\lambda \alpha_2 \| A^\frac{1}{2} v \|^2+\alpha_2\langle A^\frac{1}{2} v,A^\frac{1}{2}  f_2\rangle
+\gamma\langle A^\frac{1}{2}w, A^\frac{1}{2}z\rangle +\langle f_4, z\rangle.
\end{equation}
Subtracting the equations \eqref{Eq001AnalyR} and \eqref{Eq002AnalyR} and  taking the imaginary part  and noting that 
$$\rm{Im}\{\dual{A^\frac{1}{2}z}{A^\frac{1}{2}w}+\dual{A^\frac{1}{2}w}{A^\frac{1}{2}z}\}=0,$$
 we obtain
\begin{eqnarray}\nonumber
\gamma \lambda\alpha_ 2\|A^\frac{1}{2}v\|^2\hspace*{-0.3cm}& = &\hspace*{-0.3cm}\gamma \rm{Im} \{ \alpha_1 \langle Au, Af_1\rangle+ \langle f_3, w\rangle-\alpha_2 \langle A^\frac{1}{2}v, A^\frac{1}{2}f_2 \rangle-\langle f_4,z\rangle \}\\
\label{Eq003AnalyR}
& &+\gamma\lambda\alpha_1\|Au\|^2 +\gamma\lambda[\|z\|^2-\|w\|^2] 
\end{eqnarray}

On the other hand, 	now applying Cauchy-Schwarz and Young inequalities in \eqref{Eq003AnalyR}, using estimates   \eqref{Eq00Gevrey}$_1$, \eqref{Eq00Gevrey}$_2$ and \eqref{Eq001AnalyRN2}, we find
\begin{equation}\label{Eq104AEAnaly}
|\lambda|\alpha_2\|A^\frac{1}{2}v\|^2\hspace*{-0.1cm} \leq\hspace*{-0.1cm} C|\lambda|^{1-\theta}\{\|F\|_\mathcal{H}\|U\|_\mathcal{H}\}\quad\text{for}\quad 0<\theta < 1.
\end{equation}

Finally,  adding the estimates \eqref{Eq00Gevrey}$_1$, \eqref{Eq00Gevrey}$_2$, \eqref{Eq001AnalyRN2} and  \eqref{Eq104AEAnaly}, we find.
\begin{equation}\label{ClaseGevrey}
|\lambda|\|U\|^2_\mathcal{H}\leq C |\lambda|^{1-\theta}\{\|F\|_\mathcal{H}\|U\|_\mathcal{H}\}\quad \rm{for}\quad 0 <\theta < 1.
\end{equation}
Then, for every  $\varepsilon>0$, there exists  positive constant $K_\varepsilon$, independent of $\lambda$ such that:
$$|\lambda|\|U\|_\mathcal{H}^2\leq C|\lambda|^{1-\theta}\|F\|^2_\mathcal{H}\qquad \Longleftrightarrow \qquad \dfrac{|\lambda|^\tau \|(i\lambda I-\mathbb {B})^{-1}F\|_\mathcal{H}}{\|F\|_\mathcal{H}}\leq C$$
where  $\tau=\theta >0$ for $0<\theta<1$.  Therefore
\begin{equation}\label{Eq003Grevrey}
|\lambda|^\tau\|(i\lambda I-\mathbb {B})^{-1})^{-1}\|_{\mathcal{L}( \mathcal{H})}\leq C. 
\end{equation}

\end{proof}
So,  applying $\limsup$  when $|\lambda|\to\infty$ in \eqref{Eq003Grevrey}  of                        Theorem \ref{Theorem1.2Tebon} $S(t)$ is of  the class Gevrey  $s$, for every $s>\frac{1}{\theta}$.

Finally, of the inequalities \eqref{Eq003Grevrey} and Theorem \ref{Theorem1.2Tebon}, the inequality \eqref{DesigDef1.1} is verified and $S(t)$ is the Gevrey class $s> \frac{1}{\theta}$. Therefore, from the definition \ref {Def1.1Tebou}, the semigroups $ S (t) = e ^{\mathbb {B} t} $ is infinitely differentiable in $ \mathbb {B} $ for all $ t> 0$ and $\theta \in(0,1)$.

\begin{remark}[Gevrey Class Sharp]
The Gevrey classes determined above are Sharp, for the meaning of Sharp is given by the following theorem:
\begin{theorem}\label{TSharp}
The function $\phi(\theta)=\theta$ for $\theta\in ]0, 1[$  that determine the Gevrey classes of the semigroups $S(t)=e^{t\mathbb{B}}$ is sharp, in the sense:  If
\begin{equation}\label{Eq01Sharp}
\Phi:=\theta+\delta_0\quad {\rm for\;  all} \quad \delta_0>0 \quad{\rm such\; that}\quad  \theta+\delta_0<1 \quad{\rm and}\quad 0<\theta<1, 
\end{equation}
then
\begin{equation}\label{Eq02Sharp}
s >\dfrac{1}{\Phi }\qquad {\rm for}  \qquad  0<\theta <1, 
\end{equation}
is not a Gevrey class of the semigroup  $S(t)=e^{t\mathbb{B}}$.
\end{theorem}
 \begin{proof}
 To prove this theorem, we will use the results obtained in the  Theorem \ref{TCGevrey} and the estimates determined in the equations \eqref{Estima001NA}.  i.e, from estimative \eqref{Estima001NA},  we have
 \begin{equation*}
 |\lambda_n|^{\Phi}\|U_n\|_{\mathbb{H}}= K|\lambda_n|^{\theta+\delta_0}\|U_n\|_\mathcal{H}\geq K|\lambda_n|^{\delta_0}\to \infty,  
 \quad  {\rm when} \quad  |\lambda_n| \to \infty  
 \end{equation*}
 Therefore  $\Phi$ does not verify the \eqref{Eq1.6Tebon2020} condition of the  Theorem \ref{TCGevrey} concerning class Gevrey.
 
Then the Gevrey class  $s>\frac{1}{\theta}$  for $\theta\in ]0,1[$  the semigroup  $S(t)$ is  Sharp. 

\end{proof}
\end{remark}

{\bf Acknowledgments}
	This research was partially carried out during the visit of the first author at the Institute of Pure and Applied Mathematics (IMPA) in the 2019 summer post-doctoral program, the warm hospitality and the loan from the office of Professor Mauricio Peixoto (in memory) were greatly appreciated.  Special thanks to researcher Felipe Linares for ensuring the visit.

\bibliographystyle{amsplain}


\end{document}